\documentclass[reqno]{article}

\usepackage[margin=1in]{geometry}
\usepackage[utf8]{inputenc}
\usepackage[english]{babel}
\usepackage{csquotes}
\usepackage[final]{microtype}

\usepackage{amsmath}
\usepackage{amsfonts}
\usepackage{amsthm}
\usepackage{amssymb}
\usepackage[pdftex]{hyperref}

\usepackage{tikz-cd}
\usepackage{natbib}

\setcitestyle{numbers, open={[}, close={]}}

\newtheorem{theorem}{Theorem}[section]
\newtheorem{lemma}[theorem]{Lemma}
\newtheorem{corollary}[theorem]{Corollary}
\theoremstyle{definition}
\newtheorem{remark}[theorem]{Remark}
\newtheorem{definition}[theorem]{Definition}

\newcommand{\Q}{\mathbb{Q}}

\newcommand{\Z}{\mathbb{Z}}

\newcommand{\pri}{\mathfrak{p}}
\newcommand{\inject}{\hookrightarrow}
\newcommand{\onto}{\twoheadrightarrow}

\newcommand{\et}{\mathrm{\acute{e}t}}
\newcommand{\dR}{\mathrm{dR}}
\newcommand{\un}{\mathrm{un}}
\newcommand{\crys}{\mathrm{cr}} 
\newcommand{\loc}{\mathrm{loc}}
\newcommand{\glob}{\mathrm{glob}}

\newcommand{\tensor}{\otimes}
\newcommand{\corr}{\Rightarrow}
\newcommand{\proj}{\mathbf{P}}
\newcommand{\VV}{\mathcal{V}}

\DeclareMathOperator{\GL}{GL}
\DeclareMathOperator{\Res}{Res}
\DeclareMathOperator{\Gal}{Gal}

\DeclareMathOperator{\Jac}{Jac}

\DeclareMathOperator{\Un}{Un}

\DeclareMathOperator{\Aut}{Aut}
\DeclareMathOperator{\Iso}{Iso}
\DeclareMathOperator{\Lie}{Lie}
\DeclareMathOperator{\Spec}{Spec}
\DeclareMathOperator{\Vect}{Vect}
\DeclareMathOperator{\Sym}{Sym}
\DeclareMathOperator{\rank}{rank}
\DeclareMathOperator{\Gr}{Gr}

\begin{document}

\title{Rational Points on Solvable Curves over $\Q$ via Non-Abelian Chabauty}
\author{Jordan~S.~Ellenberg and Daniel~Rayor~Hast}
\date{March 16, 2021}

\maketitle

\begin{abstract}
We study the Selmer varieties of smooth projective curves of genus at least two defined over $\Q$ which geometrically dominate a curve with CM Jacobian. We extend a result of Coates and Kim to show that Kim's non-abelian Chabauty method applies to such a curve. By combining this with results of Bogomolov--Tschinkel and Poonen on unramified correspondences, we deduce that any cover of $\proj^1$ with solvable Galois group, and in particular any superelliptic curve over $\Q$, has only finitely many rational points over $\Q$.
\end{abstract}

\section{Introduction}
When $Y$ is a smooth projective curve of genus $g \geq 2$ over a number field $F$, Faltings' theorem~\cite{Faltings83,Faltings84} (formerly Mordell's conjecture) shows that the set $Y(F)$ of rational points is finite.

Even before Faltings, one knew finiteness of $Y(F)$ under certain conditions. One early strategy, developed by Chabauty \cite{Chabauty}, shows that $Y(F)$ is finite whenever the Mordell--Weil rank of the Jacobian $J_Y = \Jac(Y)$ is strictly smaller than $g$. More recently, Kim \cite{Kim-siegel, Kim-alb, Kim-galois, CK} developed a non-abelian version of the Chabauty method, in which the role of the Mordell--Weil group of the Jacobian is played by a $p$-adic manifold called the \emph{Selmer variety}, which we describe in \S\ref{section:kimsketch} below. If this Selmer variety has small enough dimension, one can conclude that $Y(F)$ is finite. This ``dimension hypothesis" is the nonabelian analogue of the Chabauty condition
\[
\rank J_Y(F) < g.
\]
While there are many curves that fail to satisfy the Chabauty condition, it is at least plausible to hope that every curve over every number field satisfies Kim's dimension hypothesis.

However, verifying the dimension hypothesis has been difficult, apart from certain special classes of curves. One such class, studied by Coates and Kim in \cite{CK}, is curves whose Jacobians are CM abelian varieties.

\begin{theorem}[\citet{CK}]
\label{th:ck0}
  Let $Y$ be a smooth projective curve over $\Q$ of genus at least $2$ whose Jacobian is a CM abelian variety. Then $Y$ satisfies the dimension hypothesis, and in particular, $Y(\Q)$ is finite.
\end{theorem}

The goal of the present paper is to generalize Theorem~\ref{th:ck0} to make it apply in slightly greater generality, and to show how this change can be used to substantially expand the class of curves over $\Q$ whose rational points can be proven finite via Kim's method.  In particular, we get the following corollary from our Theorem~\ref{th:main0}.

\begin{corollary}
\label{co:main0}
  Let $Y/\Q$ be a smooth superelliptic curve $y^d = f(x)$ of genus at least $2$. Then $Y(\Q)$ is finite.
\end{corollary}

Of course, this result is a special case of Faltings' theorem.  However, the non-abelian Chabauty method is a fundamentally different way of proving finiteness, whose full scope one would like to understand. In particular, Chabauty methods are more amenable than others to providing explicit upper bounds on the number of rational points~(cf.\ \cite{Col, KRZB, BD18}), as well as effective algorithms for computing the set of rational points~(cf.\ \cite{MP12, FW99, BD17, BDMTV17}). In the present work, we do not attempt to address such questions of effectiveness.

We now state our generalization of Theorem~\ref{th:ck0}.

\begin{theorem}
  \label{th:main0}
  Let $Y$ and $X$ be smooth projective curves over $\Q$ of genus at least $2$. Suppose there is a dominant map $f_K\colon Y_K \to X_K$ for some finite Galois extension $K/\Q$, and suppose the Jacobian $J_X$ of $X$ is a CM abelian variety. Then $Y(\Q)$ is finite.
\end{theorem}

\begin{remark}
  We do not quite prove the dimension hypothesis for $Y$ itself, but rather for what one might call a ``relative Selmer variety" attached to $f_K$. Our method is in some sense a non-abelian analogue of a method used in a paper of Flynn and Wetherell~\cite{FW99}.
  In that paper, the authors study certain genus $5$ curves $Y/\Q$ which are \'etale covers of a genus $2$ curve $C$, and which admit a dominant map $f$ to an elliptic curve $E$; however, the elliptic curve, whence also the map, is defined over a cubic extension $K$. The proof then proceeds by considering the map from $Y$ to the Weil restriction of scalars $\Res^K_\Q E$, which is an abelian $3$-fold defined over $\Q$; since the rank of $E(K)$ is the same as that of $(\Res^K_\Q E)(\Q)$, it suffices to show that $E(K)$ has rank less than $3$, which yields a proof of finiteness for rational points on $Y$, and finally on $C$.

  Our argument has a similar structure, utilizing the map afforded by $f$ from $Y$ to the variety $\Res^K_\Q X$. In our main application, $Y$ will be a finite \'etale cover of a curve $C$, $X$ will be a fixed genus $2$ curve, and we we use the method to prove finiteness of rational points on $Y$ and finally on $C$. We will state the ``real version'' of our theorem, a bound on the dimension of local Selmer varieties, as Theorem~\ref{thm:dim-hyp} in \S\ref{section:kimsketch} below, once we've set up the necessary definitions.
\end{remark}

\begin{remark}
  \label{rem:number-fields}
  The base field is restricted to $\Q$ due to technical difficulties in extending the non-abelian Chabauty method to curves over a number fields $F/\Q$. The method works by constructing $p$-adic Coleman functions that vanish on the set of rational points. However, the unipotent Albanese map is only $\Q_p$-analytic, so the method treats the curve as an $[F : \Q]$-dimensional $\Q_p$-manifold; thus, the vanishing of a single Coleman function does not suffice to prove Diophantine finiteness when $[F : \Q] > 1$. See \cite{Dogra-unlikely-intersections} and \cite{Hast-transcendence} for two ways to overcome this obstacle and extend the results of this paper to curves over number fields.
\end{remark}

The interest of Theorem~\ref{th:main0} would be limited without a supply of curves satisfying its conditions.  Fortunately, such a fund of examples is supplied by a theorem of Bogomolov and Tschinkel~\cite{BT02} (see also \cite{BQ17}) which shows that every hyperelliptic curve has an \'{e}tale cover which geometrically dominates a curve over $\Q$ with CM Jacobian. Poonen~\cite{P} generalized this theorem to a more general class of curves, including all superelliptic curves. By now it is well-understood that one can control the rational points of a variety $Y$ by controlling the rational points of the twists of an \'{e}tale cover of $Y$.  This circle of ideas will allow us to derive Corollary~\ref{co:main0} from Theorem~\ref{th:main0}.

In \S\ref{section:kimsketch}, we briefly sketch Kim's nonabelian Chabauty method. In \S\ref{section:fundamental-groups}, we define certain quotients of the \'etale and de Rham fundamental groups. In \S\ref{section:surjectivity}, we prove surjectivity of certain maps between fundamental groups. In \S\ref{section:selmer-varieties}, we define Selmer varieties and unipotent Albanese maps associated to the algebraic groups of \S\ref{section:fundamental-groups}, and we present the key diagram involving Selmer varieties. In \S\ref{section:bounds} and \S\ref{section:conjugation}, we prove the bounds needed for the dimension hypothesis, which we prove in \S\ref{section:dimhyp}. Finally, in \S\ref{section:superelliptic}, we combine our results with a theorem of Poonen \cite{P} to deduce finiteness of $Y(\Q)$ for several classes of curves, including hyperelliptic and superelliptic curves.

\section{A sketch of Kim's method}
\label{section:kimsketch}

Let $Y$ be a smooth projective curve of genus $g \geq 2$ over a number field $F$ such that $Y(F)$ is nonempty. The fundamental idea of the Chabauty \cite{Chabauty} method is to embed $Y$ in its Jacobian variety $J_Y$ by the Abel--Jacobi map associated to a rational point $b \in Y(F)$, then study how the $F_\pri$-points of $Y$ (where $\pri$ is a prime of $F$) interact with the $F$-points of $J_Y$. This is illustrated by the commutative diagram
\begin{equation}
  \label{eqn:chabauty}
  \begin{tikzcd}
    Y(F) \ar[r, hookrightarrow] \ar[d, hookrightarrow] & Y(F_\pri) \ar[d, hookrightarrow] \ar[dr] \\
    J_Y(F) \ar[r, hookrightarrow] & J_Y(F_\pri) \ar[r,"\log"] & \Lie(J_Y(F_\pri)).
  \end{tikzcd}
\end{equation}
By the Mordell--Weil theorem, $J_Y(F)$ is a finitely-generated abelian group of some rank $r$. In the case where $r < g$ and $\pri$ is an unramified prime of $F$ of good reduction for $Y$, Chabauty proved that the image of $Y(F_\pri)$ in the $\Q_p$-vector space $J_Y(F_\pri) \tensor_\Z \Q_p$ is dense, while the image of $J_Y(F)$ in $J_Y(F_\pri) \tensor_\Z \Q_p$ is contained in the vanishing of some nonzero form $f$; thus $Y(F)$ lies in the vanishing locus of $f$ in $Y(F_\pri)$. The right vertical arrow is locally $p$-adic analytic, given explicitly by $p$-adic integration. By the Zariski-density, $f\rvert_{Y(F_\pri)}$ is nonzero, so it has finitely many zeroes on each residue disk, whence $Y(F)$ is finite. (Actually, in this case, $f$ is a linear form, so it's enough to show that the image of $Y(F_\pri)$ in $\Lie(J_Y(F_\pri))$ isn't contained in an affine hyperplane.)

Although Faltings' result subsumes Chabauty's, this method is still of interest: Coleman \cite{Col} refined Chabauty's method to show that, when $r < g$ as above, if $\pri$ is an unramified prime of good reduction for $Y$, of residue characteristic at least $2g$, then
\[
\#Y(F) \leq N\pri + 2g(\sqrt{N\pri} + 1) - 1.
\]
Stoll \cite{Stoll06} (for hyperelliptic curves) and Katz and Zureick-Brown \cite{KZB13} (in the general case) extended Coleman's method to primes of bad reduction. In the case where $r \leq g - 3$, Stoll \cite{Stoll} (again in the hyperelliptic case) and Katz, Rabinoff, and Zureick-Brown \cite{KRZB} (in the general case) used this to obtain a uniform bound (depending only on $g$ and $[F : \Q]$) on the cardinality of $Y(F)$ for such curves. This bound is very explicit; for example, when $F = \Q$, the bound is $\#Y(\Q) \leq 84g^2 - 98g + 28$.

Unfortunately, Chabauty's method does not apply when $r \geq g$. Minhyong Kim's idea (see \cite{Kim-siegel}, \cite{Kim-alb}, \cite{Kim-galois}, \cite{CK}), motivated in large part by Grothendieck's anabelian philosophy and the section conjecture, was to develop a ``non-abelian Chabauty" method, in which the Jacobian of $Y$ is replaced by a geometric object capturing a larger piece of the fundamental group, allowing a version of Chabauty's argument to go through even when the rank of $J_Y(F)$ is large.

The description of Kim's method that follows is largely drawn from \cite{Kim-alb}.

Remaining in the abelian context for a moment, let $S$ be a finite set of primes of $F$ containing all primes of bad reduction for $Y$, let $\pri \notin S$ be a finite prime of residue characteristic $p$, and let $T$ be the union of $S$ with the set of all primes of $F$ dividing $p$ and all archimedean primes. Let $F_T/F$ be the maximal subfield of $\bar{F}$ unramified at all primes not in $T$, and let $G_T = \Gal(F_T/F)$ and $G_\pri = \Gal(\bar{F}_\pri/F_\pri)$. Then we have a cohomological version of the commutative diagram \eqref{eqn:chabauty}:
\[
\begin{tikzcd}
  Y(F) \ar[r, hookrightarrow] \ar[d] & Y(F_\pri) \ar[d] \ar[dr] \\
  H_f^1(G_T, V_p) \ar[r] & H_f^1(G_\pri, V_p) \ar[r] & \Lie(J_Y) \tensor_{\Q_p} F_\pri,
\end{tikzcd}
\]
where $V_p$ is the $\Q_p$-Tate module of $J_Y$, and $H_f^1$ denotes the pro-$p$-Selmer groups, i.e., $H_f^1(G_\pri, V_p)$ is the moduli space of crystalline $G_\pri$-torsors of $V_p$, while $H_f^1(G_T, V_p)$ is the moduli space of $G_T$-torsors of $V_p$ that are crystalline at $\pri$.

Kim's non-abelian Chabauty method replaces $V_p$, which is essentially equivalent to the abelianization of the geometric (\'etale) fundamental group of $Y$, with the $\Q_p$-pro-unipotent completion (i.e., Malcev completion) $\Pi_Y$ of the geometric fundamental group of $Y$. In order to work with schemes of finite type, we truncate after finitely many steps of the lower central series, and denote by $\Pi_{Y, n}$ the quotient of $\Pi_Y$ by the $(n + 1)$-st level of the lower central series of $\Pi_Y$.

Kim constructs moduli spaces of torsors, called \emph{Selmer varieties} of $Y$, as follows: By \cite{Kim-siegel}, the moduli space of all $G_\pri$-torsors of $\Pi_{Y, n}$ is representable by an affine $\Q_p$-variety $H^1(G_\pri, \Pi_{Y, n})$, and likewise for $H^1(G_T, \Pi_{Y, n})$. By \cite[Lemma 5]{Kim-alb}, the functor on $\Q_p$-algebras
\[
A \mapsto H^0(G_\pri, \Pi_{Y, n}(A \otimes_{\Q_p} B_{\crys})/\Pi_{Y, n}(A)),
\]
where $B_{\crys}$ is Fontaine's crystalline period ring \cite{Fontaine94}, is representable by an affine $\Q_p$-variety, denoted $H^0(G_\pri, \Pi_{Y, n}^{B_{\crys}}/\Pi_{Y, n})$. There is a connecting homomorphism
\[
H^0(G_\pri, \Pi_{Y, n}^{B_{\crys}}/\Pi_{Y, n}) \to H^1(G_\pri, \Pi_{Y, n}),
\]
the image of which is denoted $H_f^1(G_\pri, \Pi_{Y, n})$, which represents the moduli space of crystalline $G_\pri$-torsors of $\Pi_{Y, n}$, that is, torsors whose base change to $B_{\crys}$ is trivial. The global Selmer variety $H_f^1(G_T, \Pi_{Y, n})$ is then defined as the preimage of $H_f^1(G_\pri, \Pi_{Y, n})$ under the restriction map
\[
\loc_\pri\colon H^1(G_T, \Pi_{Y, n}) \to H^1(G_\pri, \Pi_{Y, n}).
\]

Likewise, $\Lie(J_Y) \tensor_{\Q_p} F_\pri$ is replaced with
\[
\Pi_Y^{\dR}/F^0 \Pi_Y^{\dR}.
\]
Here, $\Pi_Y^{\dR}$ is the de Rham fundamental group of $Y$, classifying unipotent vector bundles with integrable connection \cite{Deligne89}, and the de Rham realization $\Pi_Y^{\dR}/F^0 \Pi_Y^{\dR}$ is the moduli space of ``admissible'' torsors for $\Pi_Y^{\dR}$ with separately trivializable Hodge structure and Frobenius action \cite[Prop.\ 1]{Kim-alb}. Restricting to the Tannakian subcategory generated by bundles of unipotency class at most $n$, we obtain finite-level versions
\[
\Pi_{Y, n}^{\dR}/F^0 \Pi_{Y, n}^{\dR},
\]
which are represented by algebraic varieties over $F_\pri$.

From now on, assume $\pri$ is a prime of good reduction for $Y$ (but see \cite{Betts19} for recent progress on the bad reduction case). As in the abelian case, there are analogues of the Abel--Jacobi map---which Kim calls the (local and global) \emph{unipotent Albanese maps}---fitting into a commutative diagram
\[
\begin{tikzcd}
  Y(F) \ar[r, hookrightarrow] \ar[d] & Y(F_\pri) \ar[d] \ar[dr] \\
  H_f^1(G_T, \Pi_{Y, n}) \ar[r, "\loc_\pri"] & H_f^1(G_\pri, \Pi_{Y, n}) \ar[r, "D"] & \Res_{\Q_p}^{F_\pri}(\Pi_{Y, n}^{\dR}/F^0 \Pi_{Y, n}^{\dR}).
\end{tikzcd}
\]
The diagonal arrow in this diagram is $p$-adic analytic and can be written explicitly in terms of $p$-adic iterated integrals (generalizing the $p$-adic integrals appearing in the method of Chabauty and Coleman); the localization map $\log_\pri := D \circ \loc_\pri$ is algebraic.

Kim showed that the image of $Y(F_\pri)$ under the unipotent de Rham Albanese map is Zariski-dense in the de Rham local Selmer variety. Suppose the image of $\log_\pri$ is non-Zariski-dense; equivalently, there is an algebraic function $F$ on $\Res_{\Q_p}^{F_\pri}(\Pi_{Y, n}^{\dR}/F^0 \Pi_{Y, n}^{\dR})$ which vanishes on the image of $\log_\pri$. (Unlike in the abelian setting, $F$ might not be linear.) Since the image of $Y(F_\pri)$ is Zariski-dense, the pullback $f$ of $F$ to $Y(F_\pri)$ is nonzero. But the image of $Y(F)$ in $Y(F_\pri)$ lies in the vanishing locus of $f$, which is necessarily finite; so $Y(F)$ is finite as well.

Since $\log_\pri$ is algebraic, it suffices for the desired non-Zariski-density to prove the following ``dimension hypothesis'' for $n \gg 0$:
\begin{equation*}
  \dim H_f^1(G_T, \Pi_{Y, n}) < \dim \Res_{\Q_p}^{F_\pri}(\Pi_{Y, n}^{\dR}/F^0 \Pi_{Y, n}^{\dR})
\end{equation*}
More generally, it will suffice in our work to prove this for a certain quotient $\Psi$ of $\Pi_{Y, n}$:
\begin{equation*}
  \dim H_f^1(G_T, \Psi) < \dim \Res_{\Q_p}^{F_\pri}(\Psi^{\dR}/F^0 \Psi^{\dR})
\end{equation*}
Such a statement has been proved in several special cases; most relevant to us is the case \cite{CK} where $Y$ is projective of genus at least two, $F = \Q$, and the Jacobian of $Y$ is isogenous over $\bar{\Q}$ to a product of CM abelian varieties. (See \cite{Kim-galois} for an overview of some contexts where the method has been carried through successfully.)

In the present work, we weaken one hypothesis on $Y$: instead of requiring that $Y$ itself have CM Jacobian, we require only that $Y$ admit a dominant map onto such a curve after extension of the base number field. By combining this with theorems of Poonen \cite{P} and Bogomolov and Tschinkel \cite{BT02} on the existence of certain unramified correspondences, we deduce that any curve admitting a map to $\proj^1$ with solvable Galois group has only finitely many $\Q$-points. (See \S\ref{section:superelliptic} for the precise statement.)

\begin{theorem}
  \label{thm:dim-hyp}
  Let $Y$ and $X$ be smooth projective curves over $\Q$ of genus at least $2$ such that $Y(\Q)$ is nonempty.
  Suppose there is a dominant map $f_K\colon Y_K \to X_K$ for some finite Galois extension $K/\Q$ of degree $d$.
  Let $S$ be a finite set of primes such that $Y$ and $X$ both have good reduction outside $S$.
  Let $p$ be an odd prime not in $S$.
  Let $T = S \cup \{p, \infty\}$.
  Suppose also that there is a number field $L/K$ and a constant $B > 0$ (depending on $X$ and $T$) such that for all $n \geq 1$,
  \[
  \sum_{i=1}^{n} \dim H^2(G_{L, T}, \Gr_n(U_X)) \leq Bn^{2g - 1},
  \]
  where $G_{L, T}$ is the Galois group of the maximal algebraic extension of $L$ unramified away from $T$, we define $U_X$ to be the quotient of $\Pi_X$ by the third level of its derived series, and $\Gr_n(U_X)$ is the $n$-th graded piece of $U_X$ with respect to the lower central series filtration.
  Then for $n \gg 0$,
  \[
  \dim H_f^1(G_{\Q, T}, W_{[n]}) < \dim W_{[n]}^{\dR}/F^0 W_{[n]}^{\dR},
  \]
  where $W_{[n]}$ is the quotient of $\Pi_Y$ defined in \S\ref{section:fundamental-groups}.
  In particular, the image of $\log_p$ is non-Zariski-dense, so $Y(\Q)$ is finite.
\end{theorem}

\begin{remark}
  By \cite[Thm. 0.1]{CK}, the hypotheses of Theorem~\ref{thm:dim-hyp} are satisfied whenever the Jacobian $J_X$ of $X$ is isogenous over $\bar{\Q}$ to a product of CM abelian varieties.
\end{remark}

\section{Generalities on fundamental groups}
\label{section:fundamental-groups}
\subsection{\'Etale realizations}

This subsection and the one that follows are essentially a review of the objects constructed in \cite{Kim-alb}.

We briefly recall the definition of the $\Q_p$-pro-unipotent \'etale fundamental group. Let $X$ be a scheme over a field $F$ of characteristic zero, and let $\bar{X} := X \times_{\Spec F} \Spec \bar{F}$. Let $b \in X(F)$ be a point. Denote by $\Pi_X$ the $p$-adic unipotent completion \cite[Appendix A.2]{HM03} of the geometric \'etale fundamental group $\pi_1^{\et}(\bar{X}, b)$. Then $\Pi_X$ is a pro-unipotent group scheme over the field $\Q_p$, with an action of $G_F := \Gal(\bar{F}/F)$ induced by the action of $G_F$ on $\bar{X}$.

One can also interpret $\Pi_X$ as the fundamental group of the Tannakian category $\Un_p^{\et}(\bar{X})$ of unipotent $\Q_p$-smooth sheaves on the \'etale site of $\bar{X}$ with fiber functor $e_b$ given by the fiber at $b$. Indeed, $\Un_p^{\et}(\bar{X})$ is equivalent to the category of unipotent $\Q_p$-representations of $\pi_1^{\et}(\bar{X}, b)$, which is equivalent to the category of $\Q_p$-representations of $\Pi_X$ by the universal property of unipotent completion.

Given a geometric point $x \in X(\bar{F})$, we have the path torsor
\[
P_X^{\et}(x) = \Iso^{\tensor}(e_b, e_x).
\]
If $x \in X(F)$, then $P_X^{\et}(x)$ also has an action of $G_F$.

Given a pro-unipotent algebraic group $U$, we denote the lower central series $U^1 = U$ and $U^{n + 1} = [U, U^n]$, and the derived series by $U^{(1)} = U$ and $U^{(n + 1)} = [U^{(n)}, U^{(n)}]$. Also denote $U_n = U/U^{n + 1}$.

Let
\[
U_X := \Pi_X/\Pi_X^{(3)}
\]
be the metabelianization of $\Pi_X$. This is again a pro-unipotent $\Q_p$-group scheme with $G_F$-action.

\subsection{De Rham realizations}
We will also need the ``de Rham realizations'' of the above algebraic groups, whose definitions we now recall. Let $L$ be a field of characteristic zero (which will be a local field $F_\pri$ in our application), and let $X$ be a smooth $L$-scheme. As in \cite[\S1]{Kim-alb}, let $\Un(X)$ be the Tannakian category of unipotent vector bundles with integrable connection on $X$.

For each $n \geq 1$, let $\left<\Un_n(X)\right>$ be the full Tannakian subcategory of $\Un(X)$ generated by bundles with connection $(\VV, \nabla_\VV\colon \VV \to \Omega_{X/L} \tensor \VV)$ with index of unipotency at most $n$, i.e., such that there exists a filtration
\[
\VV = \VV_n \supseteq \VV_{n - 1} \supseteq \dots \supseteq \VV_1 \supseteq \VV_0 = 0
\]
stabilized by the connection $\nabla_\VV$ and such that each $(\VV_i/\VV_{i-1}, \nabla_\VV)$ is a trivial bundle with connection (i.e., given by pullback of a bundle with connection on $\Spec L$).

Given an $L$-scheme $S$ and a base point $b \in X(L)$, we have a morphism $b_S\colon S \to X_S$ and hence fiber functors
\[
e_b^n(S)\colon \Un_n(X \times_L S) \to \Vect_S
\]
sending each bundle $\VV$ in $\Un_n(X \times_L S)$ to the fiber $\VV_{b_S}$. Let $\Pi_{X, n}^{\dR}$ be the unipotent algebraic group over $L$ representing the functor given on $L$-schemes $S$ by
\[
S \mapsto \Aut^{\tensor}(e_b^n(S)),
\]
the group of tensor-compatible automorphisms of the functor $e_b^n(S)$. Likewise define $\Pi_X^{\dR}$ as the pro-unipotent algebraic group representing the functor $S \mapsto \Aut^{\tensor}(e_b(S))$, where $e_b(S)\colon \Un(X \times_L S) \to \Vect_S$ is the fiber functor at $b_S$.

Given another point $x \in X(L)$, we have path torsors
\[
P_{X, n}^{\dR}(x) = \Iso^{\tensor}(e_b^n, e_x^n).
\]

Now suppose $L = F_{\pri}$ is the completion of a number field $F$ at a prime $\pri$ of good reduction for $X$. As explained in \cite[p.~92]{Kim-alb}, the $F_{\pri}$-algebraic group $\Pi_{X, n}^{\dR}$ has a Hodge filtration $F^\bullet$ and, via comparison with the crystalline fundamental group (which is defined over the maximal unramified extension $L_0/\Q_p$ that is contained in $F_{\pri}$, and becomes isomorphic to the de Rham fundamental group after base change to $F_{\pri}$), an action of the Frobenius of the special fiber. We will always consider $\Pi_{X, n}^{\dR}$ with these two extra structures, which are compatible in the sense that the Frobenius operator preserves the Hodge filtration.

As in the previous section, let
\[
U_{X, n}^{\dR} := \Pi_{X, n}^{\dR}/(\Pi_{X, n}^{\dR})^{(3)}
\]
be the metabelianization of $\Pi_{X, n}^{\dR}$.

\subsection{Restriction of scalars}
Now we come to the main new construction in our paper.

Suppose we are given quasiprojective $F$-varieties $Y$ and $X$, a finite Galois extension $K/F$ of degree $d$, and a $K$-morphism $f_K\colon Y_K \to X_K$. Then the restriction of scalars $\Res_F^K X_K$ is represented by a quasiprojective variety \cite[Thm.\ 7.6.4]{BLR}. By the universal property of restriction of scalars, $f_K$ corresponds to an $F$-morphism $f\colon Y \to R := \Res_F^K X_K$. Unipotent completion is functorial, so $f$ induces a map
\[
U_f\colon U_Y \to U_R
\]
equivariant for the $G_F$-actions on $U_Y$ and $U_R$.

\begin{remark}
  \label{rem:weil-res}
  Unipotent completion commutes with finite products, so since $R_K \cong X_K^{\times d}$ (the $d$-fold direct product of $X_K$ with itself), we have a $G_K$-equivariant (but not necessarily $G_F$-equivariant) isomorphism $U_R \cong (U_X)^{\times d}$.
\end{remark}

Let $W \subseteq U_R$ be the image of $U_f$, i.e., the smallest subgroup scheme of $U_R$ through which the morphism $U_f$ factors. (This is well-defined by \cite[Thm.\ 5.39]{Milne-AG}.)  When $K$ is as small as possible, we expect that $W$ is the whole of $U_R$ in many cases, but we won't need this for our argument.

For each $n \geq 1$, let $W^{[n]} := W \cap U_R^n$ (where $U_R^n$ is the $n$-th level of the lower central series of $U_R$, and the intersection is as subgroup schemes of $U_R$), and
\[
W_{[n]} := W/W^{[n + 1]}.
\]
By construction, this induces surjections
\[
\pi_n\colon U_{Y, n} \to W_{[n]}
\]
of unipotent $\Q_p$-algebraic groups equivariant for the $G_F$-actions on $U_{Y, n}$ and $W_{[n]}$.

Similarly, for the de Rham realization, $f$ induces a map $U_f^{\dR}\colon U_Y^{\dR} \to U_R^{\dR}$ whose image we denote $W^{\dR}$. For each $n \geq 1$, let $(W^{\dR})^{[n]} := W^{\dR} \cap (U_R^{\dR})^n$ and $W_{[n]}^{\dR} := W^{\dR}/(W^{\dR})^{[n + 1]}$, and let $\pi_n^{\dR}\colon U_{Y, n}^{\dR} \to W_{[n]}^{\dR}$ be the induced surjection. These maps are compatible with the Hodge filtration and Frobenius action.

\section{Functorial properties of fundamental groups}
\label{section:surjectivity}
The following lemma about functorially induced morphisms of unipotent fundamental groups will be useful later on.

\begin{lemma}
  \label{lem:surjective}
  Let $L/\Q_p$ be a finite extension. Let $f\colon Y \to X$ be a morphism of smooth irreducible varieties over $L$ such that, for some dense open $X' \subseteq X$, the restriction $f \rvert_{f^{-1}(X')}: f^{-1}(X') \to X'$ is finite \'etale. Then the induced maps $\Pi_Y \to \Pi_X$ of $\Q_p$-unipotent \'etale fundamental groups and $\Pi_Y^{\dR} \to \Pi_X^{\dR}$ of de Rham fundamental groups are surjective.
\end{lemma}

(In fact, we will only use Lemma \ref{lem:surjective} in the case where $Y$ and $X$ are geometrically isomorphic to products of curves, in which case standard facts about fundamental groups of curves would suffice for our needs. We include the more general statement here due to the current interest in non-abelian Chabauty applied to higher-dimensional varieties, and potential application of the methods of the present paper in the higher-dimensional context.)

\begin{proof}
We proceed by comparison between the \'etale fundamental groups over an algebraic closure of $L$, the $\Q_p$-pro-unipotent fundamental groups, and the de Rham fundamental groups.

Let $y_0 \in Y$ and $x_0 \in X'$ be points such that $f(y_0) = x_0$, and let $\bar{y}_0$ and $\bar{x}_0$ be geometric points over $y_0$ and $x_0$, respectively. Let $\bar{Y}, \bar{X}, \bar{X}'$ be the base changes to $\bar{L}$. By functoriality, we have a commutative diagram of \'etale fundamental groups
\[
\begin{tikzcd}
  \pi_1^{\et}(f^{-1}(\bar{X}'), \bar{y}_0) \ar[r, "f_*", hook] \ar[d, twoheadrightarrow] & \pi_1^{\et}(\bar{X}', \bar{x}_0) \ar[d, twoheadrightarrow] \\
  \pi_1^{\et}(\bar{Y}, \bar{y}_0) \ar[r, "f_*"] & \pi_1^{\et}(\bar{X}, \bar{x}_0)
\end{tikzcd}
\]
in which the vertical maps are induced by restriction of \'etale covers to dense open subschemes. An \'etale cover of a smooth scheme is smooth, hence connected \'etale covers of smooth schemes are irreducible. Thus, the restriction of a connected cover to an dense open subscheme is connected, which implies surjectivity of the vertical maps by \cite[Expos\'e V, Prop.\ 6.9]{SGA1}. By Grothendieck's Galois theory, the image of the top horizontal map has finite index, hence the bottom horizontal map also has finite index.

Taking $\Q_p$-pro-unipotent completions \cite[Appendix A]{HM03} gives the morphism of $\Q_p$-unipotent \'etale fundamental groups $\pi_1^{\Q_p}(f)\colon \Pi_Y \to \Pi_X$, where $\Pi_Y := \pi_1^{\Q_p}(\bar{Y}, \bar{y}_0)$ and $\Pi_X := \pi_1^{\Q_p}(\bar{X}, \bar{x}_0)$. Now we need another lemma.

\begin{lemma}
  \label{lem:coker}
  Let $\varphi\colon H \to \Gamma$ be a morphism of topological groups whose image has finite index in $\Gamma$. Let $k$ be a topological field of characteristic zero. Then continuous $k$-unipotent completion induces a surjective morphism of $k$-pro-unipotent groups $\varphi^{\un}\colon H^{\un} \onto \Gamma^{\un}$.
\end{lemma}
\begin{proof}
  Let $\gamma_1, \dots, \gamma_n \in \Gamma$ be representatives of the left $H$-cosets of $\Gamma$. By construction, the image of $\Gamma$ in $\Gamma^{\un}(k)$ is Zariski-dense in $\Gamma^{\un}$. Since $\Gamma = \bigcup_{i=1}^{n} \gamma_i \cdot \varphi(H)$, taking Zariski closures, we obtain
  \[
  \Gamma^{\un} = \bigcup_{i=1}^{n} \gamma_i \cdot \overline{\varphi(H)}.
  \]
  Since $k$ has characteristic zero, $\Gamma^{\un}$ is connected, so in fact $\overline{\varphi(H)} = \Gamma^{\un}$. The image of a homomorphism of group schemes is closed, so
  \[
  \Gamma^{\un} = \overline{\varphi(H)} = \overline{\varphi^{\un}(H^{\un})} = \varphi^{\un}(H^{\un}),
  \]
  proving surjectivity of $\varphi^{\un}$.
\end{proof}

Returning to the proof of Lemma \ref{lem:surjective}: The image of $f_*$ has finite index, so Lemma \ref{lem:coker} implies $\pi_1^{\Q_p}(f)$ is surjective, proving the first claim.

By \cite[Cor.\ 4.19]{DN18}, there is a comparison isomorphism
\[
\Pi_Y^{\dR} \tensor_{L} B_{\dR} \cong \Pi_Y \tensor_{\Q_p} B_{\dR},
\]
and likewise for $X$ in place of $Y$. Since $B_{\dR}$ is faithfully flat over $\Q_p$ and $L$, surjectivity of $\Pi_Y \to \Pi_X$ implies surjectivity of the map of de Rham fundamental groups
\[
\pi_1^{\dR}(f)\colon \Pi_Y^{\dR} \onto \Pi_X^{\dR}.
\]
This completes the proof of Lemma \ref{lem:surjective}.
\end{proof}

This same map $\pi_1^{\dR}(f)$ is also given by taking de Rham fundamental groups of the pullback functor $f^*\colon \Un(X) \to \Un(Y)$ of categories of unipotent vector bundles with integrable connection.

\section{Selmer varieties and unipotent Albanese maps}
\label{section:selmer-varieties}
Here, we describe the \'etale and de Rham realizations of the Selmer varieties, summarizing the relevant parts of \cite{Kim-alb}. We also describe the unipotent Albanese maps that act as replacements for the Abel--Jacobi map.

For the remainder of the paper, $Y$ and $X$ are smooth projective curves of genus at least $2$ defined over $\Q$, and there is a dominant map $f_K\colon Y_K \to X_K$, where $K/\Q$ is a finite Galois extension of degree $d$.
Let $g$ be the genus of $X$.
Let $S$ be a finite set of rational primes such that $X$ and $Y$ both have good reduction away from $S$.
Fix an odd rational prime $p \notin S$ that splits completely in $K$, and let $T = S \cup \{p, \infty\}$.
Let $V_p := T_p J_X \tensor_{\Z_p} \Q_p$ be the $\Q_p$-Tate module of $J_X$.
Denote the absolute Galois group of any number field $L$ by $G_L$; let $G_T$ be the Galois group $\Gal(\Q_T/\Q)$, where $\Q_T$ is the maximal subfield of $\bar{\Q}$ unramified outside $T$; and fix an embedding $\bar{\Q} \inject \bar{\Q}_p$, which determines an injection $G_p := \Gal(\bar{\Q}_p/\Q_p) \inject G_\Q$.

The \'etale realization $H_f^1(G_T, U_{Y, n})$ is the moduli space of $G$-torsors for $U_{Y, n}$ that are unramified away from $T$ and crystalline at $p$, and $H_f^1(G_p, U_{Y, n})$ is the moduli space of $G_p$-torsors for $U_{Y, n}$ that are crystalline at $p$.

The de Rham realization $U_{Y, n}^{\dR}/F^0 U_{Y, n}^{\dR}$ is the moduli space of torsors for $U_{Y, n}^{\dR}$ with Frobenius structure and Hodge filtration that are \emph{admissible}, i.e., separately trivializable for the Frobenius structure and the Hodge filtration \cite[Prop.\ 1]{Kim-alb}.

For each $n \geq 1$, let $j_n^{\et, \glob}, j_n^{\et, \loc}, j_n^{\dR}$ be the unipotent Albanese maps defined in \cite{Kim-alb} as follows: Fix a base point $b \in Y(\Q)$. Then $j_n^{\et, \glob}$ (resp.\ $j_n^{\et, \loc}$) sends each $y \in Y(\Q)$ (resp.\ $y \in Y(\Q_p)$) to the class of the path torsor $[P_Y^{\et}(y)]$ in $H_f^1(G_T, U_{Y, n})$ (resp.\ $H_f^1(G_p, U_{Y, n})$). Likewise, $j_n^{\dR}$ sends each $y \in Y(\Q_p)$ to the class of the path torsor $[P_{Y, n}^{\dR}(y)]$ in $U_{Y, n}^{\dR}/F^0 U_{Y, n}^{\dR}$.

We have the following commutative diagram:
\[
\begin{tikzcd}
  Y(\Q) \ar[r, hook] \ar[d, "j_n^{\et, \glob}"'] & Y(\Q_p) \ar[d, "j_n^{\et, \loc}"'] \ar[dr, "j_n^{\dR}"] \\
  H_f^1(G_T, U_{Y, n}) \ar[r, "\loc_p"] \ar[d, "\pi_n^{\et, \glob}"'] & H_f^1(G_p, U_{Y, n}) \ar[r, "D"] \ar[d, "\pi_n^{\et, \loc}"'] & U_{Y, n}^\dR/F^0 U_{Y, n}^{\dR} \ar[d, "\pi_n^{\dR}"', twoheadrightarrow] \\
  H_f^1(G_T, W_{[n]}) \ar[r, "\loc_{p, W}"] & H_f^1(G_p, W_{[n]}) \ar[r, "D"] & W_{[n]}^{\dR}/F^0 W_{[n]}^{\dR}
\end{tikzcd}
\]
Here, the vertical maps $\pi_n^{\et, \glob}$, $\pi_n^{\et, \loc}$, and $\pi_n^{\dR}$ are functorially induced by $\pi_n$. As proved in \cite{Kim-alb}, if the image of the algebraic map $\log_{p, W} := D \circ \loc_{p, W}$ (and hence the image of $\log_p := D \circ \loc_p$) is not Zariski-dense, then $Y(\Q)$ is finite.

\section{Bounds for local and global Selmer varieties}
\label{section:bounds}

With the basic setup now in place, we turn to our main goal, which is to show that the image of the global Selmer variety in the de Rham local Selmer variety is not Zariski dense.  This has two parts:  giving an upper bound for the dimension of the global Selmer variety (filling the role played by the Mordell--Weil group in classical Chabauty) and giving a lower bound for the dimension of the de Rham local Selmer variety (filling the role classically played by the genus.)

We begin with a few remarks about the structure of $W$.
The canonical surjection $R_K \onto X_K$ induces a $G_K$-equivariant map $U_{R, n} \to U_{X, n}$. Restricting this to $W_{[n]}$ yields a map
\[
W_{[n]} \to U_{X, n}
\]
which is functorially induced by $f_K\colon Y_K \to X_K$, hence is surjective by Lemma \ref{lem:surjective}.
Also, as mentioned in Remark \ref{rem:weil-res}, there is a $G_K$-equivariant isomorphism of $\Q_p$-algebraic groups
\[
U_{R, n} \cong U_{X, n}^{\times d}.
\]

We can also take graded pieces: define $Z_{[n]} := W^{[n]}/W^{[n + 1]}$, which fits into an exact sequence
\[
0 \to Z_{[n]} \to W_{[n]} \to W_{[n - 1]} \to 0
\]
of $\Q_p[G_\Q]$-modules. Similarly, write $Z_n(U_R)$, $Z_n(U_Y)$, and $Z_n(U_X)$ for the $n$-th graded pieces (with respect to the lower central series) of the other algebraic groups. By construction of $W_{[n]}$, we obtain as above a $G_\Q$-equivariant injection
\[
Z_{[n]} \inject Z_n(U_R)
\]
and a $G_K$-equivariant surjection
\[
Z_{[n]} \onto Z_n(U_X).
\]

\subsection{Lower bounds for the de Rham local Selmer variety}
In order to obtain lower bounds for the dimension of $W_{[n]}^{\dR}/F^0 W_{[n]}^{\dR} $, we will need upper bounds for the dimension of the filtered piece $F^0 W_{[n]}^{\dR}$.

\begin{lemma}
  \label{lem:F0}
  There is a constant $A$ (depending only on $X$, $T$, and $d$) such that, for all $n \geq 1$,
  \[
  \dim F^0 W_{[n]}^{\dR} \leq An^g.
  \]
\end{lemma}
\begin{proof}
  Since the $p$-adic Hodge filtration is defined over $\Q_p$, which contains $K$ (because we chose $p$ to be completely split in $K$), the $G_K$-equivariant inclusion
  \[
  W_{[n]}^{\dR} \inject (U_{X, n}^{\dR})^{\times d}
  \]
  is compatible with the Hodge structures, hence induces an inclusion of Hodge filtered pieces
  \[
  F^0 W_{[n]}^{\dR} \inject F^0 (U_{X, n}^{\dR})^{\times d} = (F^0 U_{X, n}^{\dR})^{\times d}.
  \]
  By the proof of \cite[Thm.\ 2]{CK}, there is a constant $A'$ such that $\dim F^0 U_{X, n}^{\dR} \leq A'n^g$ for all $n \geq 1$. Thus,
  \[
  \dim F^0 W_{[n]}^{\dR} \leq d \cdot \dim F^0 U_{X, n}^{\dR} \leq dA'n^g.
  \]
\end{proof}

\subsection{Upper bounds for the global Selmer variety}
In this section, we show that the dimension of the global cohomology space $H^1(G_T, W_{[n]})$ can be bounded by the dimension of ``abelianizations'' coming from the graded pieces $Z_{[n]}$, which can in turn be studied using the Euler characteristic formula.

\begin{lemma}
  \label{lem:H0}
  For all $n \geq 1$, $H^0(G_T, Z_{[n]}) = 0$.
\end{lemma}
\begin{proof}
  The injection $Z_{[n]} \inject Z_n(U_R)$ induces an injection
  \[
  H^0(G_T, Z_{[n]}) \inject H^0(G_T, Z_n(U_R)).
  \]
  As explained in \cite[\S3]{CK}, $Z_n(U_X)$ has Frobenius weight $n$. Since $p$ splits completely in $K$, we have $H^0(G_{K, T}, Z_n(U_X)) = 0$ and a Frobenius-equivariant injection
  \[
  H^0(G_T, Z_n(U_R)) \inject H^0(G_{K, T}, Z_n(U_R)) \cong H^0(G_{K, T}, Z_n(U_X))^{d} = 0.
  \]
\end{proof}

\begin{lemma}
  \label{lem:abelianize}
  For all $n \geq 1$,
  \[
  \dim H^1(G_T, W_{[n]}) \leq \sum_{i=1}^{n} \dim H^1(G_T, Z_{[i]}).
  \]
\end{lemma}
\begin{proof}
  By Lemma \ref{lem:H0}, $H^0(G_T, W_{[n]}) \subseteq H^0(G_T, W_{[n - 1]})$ for each $n \geq 2$, so $H^0(G_T, W_{[n]}) \subseteq H^0(G_T, W_{[1]}) = H^0(G_T, Z_{[1]}) = 0$. By the claim in the proof of \cite[Prop.\ 2]{Kim-siegel} (noting that the proof works without modification for the filtration $W^{[n]}$, not just for the lower central series), the exact sequence
  \[
  0 \to Z_{[n]} \to W_{[n]} \to W_{[n - 1]} \to 0
  \]
  induces an exact sequence
  \[
  0 = H^0(G_T, W_{[n - 1]}) \to H^1(G_T, Z_{[n]}) \to H^1(G_T, W_{[n]}) \to H^1(G_T, W_{[n - 1]}) \overset{\delta}{\to} H^2(G_T, Z_{[n]})
  \]
  in the sense that $H^1(G_T, W_{[n]})$ is a $H^1(G_T, Z_{[n]})$-torsor over the subvariety of $H^1(G_T, W_{[n-1]})$ given by the kernel of the boundary map $\delta$. Thus,
  \[
  \dim H^1(G_T, W_{[n]}) \leq \dim H^1(G_T, W_{[n-1]}) + \dim H^1(G_T, Z_{[n]}),
  \]
  and the result follows by induction.
\end{proof}

\begin{lemma}
  \label{lem:H3+}
  For all $r \geq 3$, we have $H^r(G_T, Z_{[n]}) = 0$.
\end{lemma}
\begin{proof}
  Let $\Lambda \subset Z_{[n]}$ be a $G_T$-equivariant lattice, i.e., a $\Z_p[G_T]$-submodule with $\Lambda \otimes_{\Z_p} \Q_p \cong Z_{[n]}$. Then $M = Z_{[n]}/\Lambda$ is a discrete torsion $G_T$-module. By \cite[Prop.\ 8.3.18]{NSW08}, the $p$-cohomological dimension of $G_T$ is at most $2$ for $p$ odd, so $H^r(G_T, M) = 0$.

  By \cite[Ch.\ I, Cor.\ 4.15]{Milne-ADT}, the profinite group $G_T$ satisfies condition $(F_p)$ of \cite[Remark 1]{Jannsen89}, so we also have $H^r(G_T, \Lambda) = \varprojlim_j H^r(G_T, \Lambda/p^j \Lambda) = \varprojlim_j 0 = 0$. Applying cohomology to the short exact sequence $0 \to \Lambda \to Z_{[n]} \to M \to 0$, we obtain $H^r(G_T, Z_{[n]}) = 0$.
\end{proof}

Combining Lemmas \ref{lem:H0} and \ref{lem:H3+} with the Euler characteristic formula \cite[Lemma 2]{Jannsen89}, for each $n \geq 1$, we obtain
\[
\dim H^1(G_T, Z_{[n]}) = \dim H^2(G_T, Z_{[n]}) + \dim Z_{[n]} - \dim Z_{[n]}^{+},
\]
where $Z_{[n]}^{+}$ denotes the positive eigenspace of the action of complex conjugation on $Z_{[n]}$. So to bound $H^1(G_T, Z_{[n]})$, we can work instead with $H^2(G_T, Z_{[n]})$ and $Z_{[n]}^{+}$.

We now turn to the problem of bounding the dimension of $H^2(G_T, Z_{[n]})$.

\begin{lemma}
  \label{lem:H2}
  Suppose there is a number field $L/K$ and a constant $B$ (depending only on $X$ and $T$) such that, for all $n \neq 1$,
  \[
  \sum_{i=1}^{n} \dim H^2(G_{L, T}, Z_{i}(U_X)) \leq Bn^{2g - 1}.
  \]
  Then for all $n \geq 1$,
  \[
  \sum_{i=1}^{n} \dim H^2(G_T, Z_{[i]}) \leq dBn^{2g - 1},
  \]
  where $d = [K : \Q]$.
\end{lemma}
\begin{proof}
  Since $G_{L, T} \leq G_T$ is a subgroup of some finite index $m$, we have a corestriction map
  \[
  H^2(G_{L, T}, Z_{[i]}) \to H^2(G_T, Z_{[i]})
  \]
  for each $i$. Precomposing corestriction with restriction yields multiplication by $m$. Since $Z_{[i]}$ is a divisible abelian group, multiplication by $m$ is an automorphism, so the above corestriction map is surjective, and it suffices to bound $\dim H^2(G_{L, T}, Z_{[i]})$.

  By the semisimplicity theorem of Faltings and Tate, the $\Q_p$-Tate module $Z_1(U_X) = V_p = T_p J_X \tensor_{\Z_p} \Q_p$ is a semisimple $G_L$-representation. Hence, $(V_p^{\tensor i})^{\oplus d}$ is semisimple. As in \cite{CK}, we have a surjection $V_p^{\tensor i} \onto Z_i(U_X)$ which splits, realizing $Z_i(U_X)$ as a direct summand of $V_p^{\tensor i}$.

  We also have an inclusion of $\Q_p[G_{L, T}]$-modules
  \[
  Z_{[i]} \inject Z_i(U_X)^{\oplus d}.
  \]
  But $Z_i(U_X)^{\oplus d}$ is semisimple, so $Z_{[i]}$ is in fact a direct summand. Since cohomology preserves direct summands, it follows that
  \[
  \dim H^2(G_{L, T}, Z_{[i]}) \leq d \cdot \dim H^2(G_{L, T}, Z_i(U_X)).
  \]
  Summing the above over all $1 \leq i \leq n$, we are done.
\end{proof}

\begin{remark}
  Suppose the Jacobian $J_X$ of $X$ is isogenous over $\bar{\Q}$ to a product of CM abelian varieties. Then there is a number field $L/K$ such that:
  \begin{itemize}
    \item The complex multiplication and decomposition of $J_X$ are defined over $L$;
    \item $K(J_X[p]) \subseteq L$; and
    \item the image $\Gamma$ of the associated Galois representation
      \[
      \rho_L: \Gal(\bar{L}/L) \to \GL(V_p)
      \]
      is isomorphic to $\Z_p^r$ for some integer $r$.
  \end{itemize}
  By \cite[Thm.\ 1]{CK}, there is a constant $B$ (depending on $X$ and $T$) such that
  \[
  \sum_{i=1}^{n} \dim H^2(G_{L, T}, Z_i(U_X)) \leq Bn^{2g - 1},
  \]
  so the conclusion of Lemma \ref{lem:H2} holds. In this case, $V_p$ is isomorphic as a $G_L$-representation to a direct sum of characters, so we also do not need to appeal to the semisimplicity theorem in the proof of Lemma \ref{lem:H2}. (In particular, the results of \S\ref{section:superelliptic} do not rely on the semisimplicity theorem.)
\end{remark}

\section{Bounding the invariants of complex conjugation}

\label{section:conjugation}

We need a lower bound on the dimension $Z_{[i]}^{+}$, the subgroup of $Z_{[i]}$ on which complex conjugation acts as the identity. We begin with a combinatorial lemma.

\begin{lemma}
  \label{lem:symm-powers}
  Let $k$ be a field not of characteristic $2$. Let $V$ be a $k$-vector space of finite dimension $m$. Let $c\colon V \to V$ be a linear involution which is not multiplication by $\pm 1$. Then
  \[
  \lim_{n \to \infty} \frac{\dim \Sym^n(V)^{+}}{\dim \Sym^n(V)} = \frac{1}{2}.
  \]
\end{lemma}

\begin{proof}
  Define $a_n = \dim \Sym^n(V)^{+} - \dim \Sym^n(V)^{-}$. Since $\dim \Sym^n(V) = \binom{n + m - 1}{m - 1}$, which is a polynomial in $n$ of degree $m - 1$, it suffices to show that to show that $a_n = O(n^{m - 2})$. Write $m^+, m^-$ for the dimension of the $+1$ and $-1$ eigenspaces of $c$ on $V$. Then we have a generating function identity
  \[
  \sum_{n=0}^{\infty} a_n t^n = (1 - t)^{-m^+} (1 + t)^{-m^{-}} = \sum_{i=1}^{m^+} \frac{A_i}{(1 - t)^i} + \sum_{i=1}^{m^-} \frac{B_i}{(1 + t)^i},
  \]
  where $A_i, B_i \in \Q$ are given by partial fraction decomposition. The power series coefficients of $(1 \pm t)^{-i}$ are of order $O(n^{i - 1})$, so $a_n = O(n^{\max\{m^+, m^-\} - 1})$. Since $\max\{m^+, m^-\} < m$, we have $a_n = O(n^{m - 2})$, completing the proof.
\end{proof}

\begin{remark}
  Lemma \ref{lem:symm-powers} is equivalent to the following combinatorial statement: Suppose we have $a \geq 1$ labelled bins colored blue, $b \geq 1$ labelled bins colored green, and $n$ indistinguishable balls. Then for $n \gg 0$, out of all the ways to place the $n$ balls into the $a + b$ bins, approximately half result in the total number of balls in green bins being even.

  Coates and Kim \cite[proof of Cor.\ 0.2, pp.\ 847--848]{CK} prove the special case of Lemma \ref{lem:symm-powers} where $a = b$, in which case there are additional symmetries that simplify the problem.
\end{remark}

\begin{lemma}
  \label{lem:Z+}
  There is a constant $C > 0$ (depending only on $X$ and $T$) such that, for all $n \geq 1$,
  \[
  \sum_{i=1}^{n} \dim Z_{[i]}^{+} \geq Cn^{2g}.
  \]
\end{lemma}

\begin{proof}
  Let $c \in G_\Q$ be a complex conjugation.  We may choose a $\Q_p$-basis $\bar{f}_1, \dots, \bar{f}_{2g(Y)}$ of  $Z_1(U_Y)$ (in other words, of the $\Q_p$-Tate module of $Y$) such that $c(\bar{f}_i) = \pm\bar{f}_i$ for each $i = 1, \dots, 2g(Y)$.  Let $\tilde{f}_i$ be lifts of $\bar{f}_i$ to $U_Y$. Recall the $G_\Q$-equivariant map
  \[
  U_f\colon U_Y \to W.
  \]
  Let $f_i := U_f(\tilde{f}_i)$.

  We also have a surjective homomorphism $W \onto U_X$.  Recall that the Galois action on $W$ was defined to be the one inherited from the Galois action on $U_Y$; since the map $Y \to X$ is defined only over $K$, the map $U_Y \to U_X$, whence the map $W \to U_X$, is equivariant only for the Galois group $G_K$, not the whole of $G_\Q$.  In particular, if $K$ is not totally real, the map from $W$ to $U_X$ is not equivariant for complex conjugation.  We will see that this doesn't matter; using the purely combinatorial Lemma~\ref{lem:symm-powers}, we can lower-bound the conjugation-invariant part of $Z_{[i]}$ using only the group structure of $W$, no matter what the action of $c$ is, as we now explain.

  After reordering $f_1, \ldots, f_{2g(Y)}$ if necessary, we may assume that $f_1, \ldots, f_{2g}$ project to a basis $a_1,\ldots, a_{2g}$ of $Z_1(U_X)$.  The projection of $f_i$ to $Z_1(W)$ is an eigenvector for $c$; again reordering if necessary, we may assume the eigenvalue is $+1$ for $i = 1,\ldots, s$ and $-1$ for $i=s+1, \ldots, 2g$. 

  Let $L \subseteq \Lie(W)$ be the Lie subalgebra generated by $f_1, \dots, f_{2g}$, and let $L^{n}$ denote the $n$-th level of the lower central series of $L$. Let $L_{[n]}$ be the subspace of $Z_{[n]}$ generated by $L^{n}$. The homomorphism $W \onto U_X$ induces a homomorphism $L_{[n]} \onto Z_n(U_X)$.  By  \cite[proof of Cor.\ 0.2, p.\ 847]{CK}, a basis for $Z_n(U_X)$ when $n \geq 2$ is given by elements of the form 
  \[
  [ [ \dots [ [a_{i_1}, a_{i_2}], a_{i_3}], \dots], a_{i_n}]
  \]
  with $i_1 < i_2$ and $i_2 \geq i_3 \geq \ldots \geq i_n$.  In particular, this means the elements
  \[
  [ [ \dots [ [f_{i_1}, f_{i_2}], f_{i_3}], \dots], f_{i_n}]
  \]
  are linearly independent in $L_{[n]}$.  Note that $c$ acts on such an element by $\pm 1$; more precisely, it acts as $(-1)^k$ where $k$ is the number of the $i_1, \ldots, i_n$ which are greater than $s$.  Write $V_n <  L_{[n]}$ for the space spanned by the elements above.  

  We now consider three cases. 

  If $s=2g$, then $c$ acts as $1$ on $V_n$.  So $\dim V_n^{+} = \dim L_{[n]}$, which is bounded below by $C n^{2g-1}$ for all $n$.  This proves the lemma in this case.

  If $s=0$, then $c$ acts as $(-1)^n$ on $V_n$.  So  $\dim V_n^{+} = \dim L_{[n]}$ for all even $n$, and the lemma follows again.

  Now suppose $0 < s < 2g$.  Consider the space $V_{1,n} < V_n$ spanned by
  \[
  [ [ \dots [ [f_{1}, f_{i_2}], f_{i_3}], \dots], f_{i_n}].
  \]
  This basis for $V_{1,n}$ is naturally in bijection with the set of monomials $x_{i_2} \ldots x_{i_n}$ of degree $n-1$ in the variables $x_1, x_2, \ldots, x_{2g}$, excluding $x_1^{n - 1}$; thus $\dim V_{1,n}$ is on order $n^{2g-1}$ as $n$ grows.  A basis for $V_{1,n}^+$ is given by those monomials whose total degree in $x_{s+1}, \ldots, x_{2g}$ is even.  Lemma~\ref{lem:symm-powers} tells us precisely that
  \[
  \dim V_{1,n}^+ = (1/2) \dim V_{1,n} + o(\dim V_{1,n}) \geq C n^{2g-1}
  \]
  Once again, the lemma follows.
\end{proof}

\section{Proof of Theorem~\ref{thm:dim-hyp}}
\label{section:dimhyp}

Our goal is to prove that, for $n$ sufficiently large,
\[
\dim H_f^1(G_T, W_{[n]}) < \dim W_{[n]}^{\dR}/F^0 W_{[n]}^{\dR}.
\]

Recall that $H_f^1(G_T, W_{[n]})$ is a subvariety of $H^1(G_T, W_{[n]})$; it thus suffices to bound the dimension of the latter variety, which by Lemma \ref{lem:abelianize} and the Euler characteristic formula is bounded above as follows:
\begin{align*}
  \dim H^1(G_T, W_{[n]}) &\leq \sum_{i=1}^{n} \left( \dim Z_{[i]} + \dim H^2(G_T, Z_{[i]}) - \dim Z_{[i]}^{+} \right) \\
  &= \dim W_{[n]} + \sum_{i=1}^{n} \dim H^2(G_T, Z_{[i]}) - \sum_{i=1}^{n} \dim Z_{[i]}^{+}.
\end{align*}

By Lemma~\ref{lem:H2}, the contribution of $\sum_{i=1}^{n} \dim H^2(G_T, Z_{[i]})$ is $O(n^{2g-1})$.  By Lemma~\ref{lem:Z+}, we know $\sum_{i=1}^{n} \dim Z_{[i]}^{+}$ is bounded below by a multiple of $n^{2g}$.  Putting these facts together, we have
\[
\dim H^1(G_T, W_{[n]}) \leq \dim W_{[n]} - C n^{2g}
\]
for some $C > 0$.

On the other hand, by Lemma~\ref{lem:F0}, we have
\[
\dim F^0 W_{[n]}^{\dR} = O(n^g)
\]

Thus, for $n$ large enough, we have
\[
\dim H^1_f(G_T, W_{[n]}) \leq \dim H^1(G_T, W_{[n]}) < \dim W_{[n]} - \dim F^0 W_{[n]}^{\dR} = \dim W_{[n]}^{\dR}/F^0 W_{[n]}^{\dR}
\]
which is the desired result.

\begin{remark}
  The difficulty in applying the same technique over a number field $F$ other than $\Q$ is that
  \[
  \dim Z_{[n]} - \dim Z_{[n]}^{+}
  \]
  in the Euler characteristic formula is replaced with
  \[
  \sum_{v \text{ real}} (\dim Z_{[n]} - \dim Z_{[n]}^{+, v}) + \sum_{v \text{ complex}} \dim Z_{[n]},
  \]
  where the sums are over the real and complex places of $F$, and $Z_{[n]}^{+, v}$ is the $1$-eigenspace of the complex conjugation associated to a real place $v$.  For our argument to go through, this sum needs to be strictly smaller than $\dim Z_{[n]}$.  Obviously this is impossible if $F$ has a complex place, and even if $F$ is totally real, the summands corresponding to real places should have size about $(1/2) \dim Z_{[n]}$ for large $n$, which blocks the method from working when $F$ is larger than $\Q$. (This is not merely an artefact of the ``abelianization'' that replaces $W$ with its graded pieces; the long exact sequence in the proof of Lemma \ref{lem:abelianize} shows that this abelianization can add at most $\sum_{i=1}^{n} \dim H^2(G_{F, T}, Z_{[n]})$ to the dimension, so $\dim H^1(G_{F, T}, W_{[n]})$ is still too large.)

  As noted in Remark \ref{rem:number-fields}, this obstacle can be overcome by a different approach that essentially involves ``trading degree for dimension'' via restriction of scalars and then proving $p$-adic functional transcendence theorems to extend the method to higher-dimensional varieties.
\end{remark}

\section{Application to superelliptic curves}
\label{section:superelliptic}
We now combine Theorem \ref{thm:dim-hyp} with results of Poonen \cite{P} and Bogomolov and Tschinkel \cite{BT02} to prove finiteness of $C(\Q)$ whenever $C$ is a smooth proper curve over $\Q$ of genus at least $2$ such that there exists a map $C \to \proj^1$ with solvable automorphism group (e.g., when $C$ is superelliptic).

\begin{definition}
  Given varieties $X$ and $Y$ over a field $k$, we write $X \corr Y$ if, over $\bar{k}$, there exists an \'etale cover $Z \to X$ and a dominant morphism $Z \to Y$.
\end{definition}

\begin{theorem}
  \label{thm:curve-finite}
  Let $C$ be a smooth proper curve over $\Q$ such that $C \corr X$ for some curve $X$ over $\Q$ such that the Jacobian of $X$ has potential CM. Then $C(\Q)$ is finite.
\end{theorem}

\begin{proof}
  Since $C \corr X$, there exists an \'etale cover $\pi\colon Y \to C$ and a dominant morphism $f\colon Y \to X$ defined over $\bar{\Q}$, and hence over some Galois extension $L/\Q$ of finite degree $m$. Let $Y_1, \dots, Y_m$ be the Galois conjugates of $Y$, and let $Z$ be a connected component of the $\Q$-scheme $\Gal(L/\Q) \backslash (Y_1 \sqcup \dots \sqcup Y_m)$. Then $Z$ is a connected (but not necessarily geometrically connected) finite \'etale cover of $C$ defined over $\Q$. Let $\tilde{Z} \to Z \to C$ be the Galois closure of the connected \'etale cover $Z \to C$.


  By the Chevalley--Weil theorem \cite[\S4.2]{Serre-MW} there is a finite set of twists $\tilde{Z}^\tau/\Q$, each with a $\Q$-rational map $\pi^\tau$ to $C$, such that $C(\Q)$ is covered by $\pi^\tau(\tilde{Z}^\tau(\Q))$. (See e.g.\ \cite[p.\ 2]{Voloch12}.) Let $\{W_1, \dots, W_\nu\}$ be the set of all connected components of the twists $\tilde{Z}^\tau$ such that $W_i(\Q)$ is nonempty. By \cite[Cor.\ 4.5.14]{EGA-IV.2}, each $W_i$ is geometrically connected.

  Each $W_i$ is a smooth, proper, geometrically connected curve over $\Q$ with a dominant $\bar{\Q}$-morphism to $X$. Because $X$ has potential CM, the theorem of Coates--Kim \cite[Thm.\ 0.1]{CK}) implies that the hypotheses of Theorem \ref{thm:dim-hyp} hold for each $W_i$. Hence, $W_i$ is finite for all $i$, so $C(\Q) \subseteq \bigcup_{i=1}^{\nu} \pi^\tau(W_i(\Q))$ is also finite.
\end{proof}

Let $C_6$ be the smooth projective model of the curve with affine equation $y^2 = x^6 - 1$.

\begin{theorem}[\citet{BT02}; \citet{P}, Thm.\ 1.7, Cor.\ 1.8]
  \label{thm:poonen-corr}
  Let $C$ be a curve of genus $g(C) \geq 2$ over an algebraically closed field of characteristic zero. Let $G$ be a subgroup of $\Aut(C)$. Let $D = C/G$, and assume $D$ is of genus $g(D) \leq 2$. Suppose also that at least one of the following holds:
  \begin{enumerate}
    \item $g(D) \in \{1, 2\}$.
    \item $G$ is solvable.
    \item There are two distinct points of $D$ above which the ramification indices are not coprime.
    \item There are three points of $D$ above which the ramification indices are divisible by $2, 3, \ell$, respectively, where $\ell$ is a prime with either $\ell \leq 89$ or
      \[
      \ell \in \{101, 103, 107, 131, 167, 191\}.
      \]
  \end{enumerate}
  Then $C \corr C_6$.
\end{theorem}

\begin{corollary}
  Let $C$ be a smooth projective curve over $\Q$ of genus at least $2$. Suppose $C_{\bar{\Q}}$ satisfies the hypotheses of Theorem \ref{thm:poonen-corr}. Then $C(\Q)$ is finite.  In particular, if $C$ is a smooth superelliptic curve $y^d = f(x)$ of genus at least two, $C(\Q)$ is finite.
\end{corollary}
\begin{proof}
  Immediate from Theorem~\ref{thm:curve-finite}, Theorem~\ref{thm:poonen-corr}, and the fact that the Jacobian of $C_6$ is isogenous to the product of elliptic curves $y^2 = x^3 - 1$ and $y^2 = x^3 + 1$, which both have CM by the ring of integers of $\Q(\sqrt{-3})$.
\end{proof}

\begin{remark}
  In \cite[Conj.\ 3.1]{BT-small-fields}, Bogomolov and Tschinkel conjecture that $C \corr C'$ for \emph{any} smooth projective curves $C$ and $C'$ over $\bar{\Q}$ of genus at least $2$. If this conjecture is true, our method applies to show finiteness of $C(\Q)$ for every such curve $C$ over $\Q$.
\end{remark}

\begin{remark}
  If there exists a twist $X'$ of $C_6$ and an unramified correspondence $C \corr X'$ defined over $\Q$ (rather than merely over $\bar{\Q}$, as provided by Theorem \ref{thm:poonen-corr}), then finiteness of $C(\Q)$ follows immediately from \cite{CK}. However, we do not see any reason to expect such a twist to exist in general.
\end{remark}

\section*{Funding}
This work was supported by the National Science Foundation [DMS-1402620, RTG-1502553]; a Guggenheim Fellowship [to J.S.E.]; and the Simons Foundation [Simons Collaboration on Arithmetic Geometry, Number Theory, and Computation, to D.R.H.].

\section*{Acknowledgements}
The authors thank Minhyong Kim, Jackson Morrow, Isabel Vogt, and several anonymous reviewers for their helpful comments.

\bibliography{EllenbergHastSolvableNonabelianChabauty}

\begin{thebibliography}{37}
\providecommand{\natexlab}[1]{#1}
\providecommand{\url}[1]{\texttt{#1}}
\expandafter\ifx\csname urlstyle\endcsname\relax
  \providecommand{\doi}[1]{doi: #1}\else
  \providecommand{\doi}{doi: \begingroup \urlstyle{rm}\Url}\fi

\bibitem[SGA(2003)]{SGA1}
\emph{Rev\^{e}tements \'{e}tales et groupe fondamental ({SGA} 1)}, volume~3 of
  \emph{Documents Math\'{e}matiques (Paris)}.
\newblock Soci\'{e}t\'{e} Math\'{e}matique de France, Paris, 2003.
\newblock ISBN 2-85629-141-4.
\newblock S\'{e}minaire de g\'{e}om\'{e}trie alg\'{e}brique du Bois Marie
  1960--61. Directed by A. Grothendieck, with two papers by M. Raynaud. Updated
  and annotated reprint of the 1971 original [Lecture Notes in Math., 224,
  Springer, Berlin].

\bibitem[Balakrishnan et~al.(2019)Balakrishnan, Dogra, M\"{u}ller, Tuitman, and
  Vonk]{BDMTV17}
J.~Balakrishnan, N.~Dogra, J.~S. M\"{u}ller, J.~Tuitman, and J.~Vonk.
\newblock Explicit {C}habauty-{K}im for the split {C}artan modular curve of
  level 13.
\newblock \emph{Ann. of Math. (2)}, 189\penalty0 (3):\penalty0 885--944, 2019.
\newblock ISSN 0003-486X.
\newblock \doi{10.4007/annals.2019.189.3.6}.

\bibitem[Balakrishnan and Dogra(2019)]{BD18}
J.~S. Balakrishnan and N.~Dogra.
\newblock An effective {C}habauty-{K}im theorem.
\newblock \emph{Compos. Math.}, 155\penalty0 (6):\penalty0 1057--1075, 2019.
\newblock ISSN 0010-437X.
\newblock \doi{10.1112/s0010437x19007243}.

\bibitem[Balakrishnan and Dogra(2020)]{BD17}
J.~S. Balakrishnan and N.~Dogra.
\newblock Quadratic {C}habauty and rational points {II}: Generalised height
  functions on {S}elmer varieties.
\newblock \emph{Int. Math. Res. Not. IMRN}, 02 2020.
\newblock ISSN 1073-7928.
\newblock \doi{10.1093/imrn/rnz362}.

\bibitem[Betts(2019)]{Betts19}
L.~A. Betts.
\newblock The motivic anabelian geometry of local heights on abelian varieties.
\newblock \emph{ArXiv e-prints}, 2019.
\newblock arXiv:1706.04850v2 [math.NT].

\bibitem[Bogomolov and Qian(2017)]{BQ17}
F.~Bogomolov and J.~Qian.
\newblock On contraction of algebraic points.
\newblock \emph{Bull. Korean Math. Soc.}, 54\penalty0 (5):\penalty0 1577--1596,
  2017.
\newblock ISSN 1015-8634.
\newblock \doi{10.4134/BKMS.b160680}.

\bibitem[Bogomolov and Tschinkel(2002)]{BT02}
F.~Bogomolov and Y.~Tschinkel.
\newblock Unramified correspondences.
\newblock In \emph{Algebraic Number Theory and Algebraic Geometry}, volume 300
  of \emph{Contemp. Math.}, pages 17--25. Amer. Math. Soc., Providence, RI,
  2002.
\newblock \doi{10.1090/conm/300/05141}.

\bibitem[Bogomolov and Tschinkel(2007)]{BT-small-fields}
F.~Bogomolov and Y.~Tschinkel.
\newblock Algebraic varieties over small fields.
\newblock In \emph{Diophantine Geometry}, volume~4 of \emph{CRM Series}, pages
  73--91. Ed. Norm., Pisa, 2007.

\bibitem[Bosch et~al.(1990)Bosch, L\"{u}tkebohmert, and Raynaud]{BLR}
S.~Bosch, W.~L\"{u}tkebohmert, and M.~Raynaud.
\newblock \emph{N\'{e}ron Models}, volume~21 of \emph{Ergebnisse der Mathematik
  und ihrer Grenzgebiete (3) [Results in Mathematics and Related Areas (3)]}.
\newblock Springer-Verlag, Berlin, 1990.
\newblock ISBN 3-540-50587-3.
\newblock \doi{10.1007/978-3-642-51438-8}.

\bibitem[Chabauty(1941)]{Chabauty}
C.~Chabauty.
\newblock Sur les points rationnels des courbes alg\'{e}briques de genre
  sup\'{e}rieur \`a l'unit\'{e}.
\newblock \emph{C. R. Acad. Sci. Paris}, 212:\penalty0 882--885, 1941.
\newblock ISSN 0001-4036.

\bibitem[Coates and Kim(2010)]{CK}
J.~Coates and M.~Kim.
\newblock Selmer varieties for curves with {CM} {J}acobians.
\newblock \emph{Kyoto J. Math.}, 50\penalty0 (4):\penalty0 827--852, 2010.
\newblock ISSN 2156-2261.
\newblock \doi{10.1215/0023608X-2010-015}.

\bibitem[Coleman(1985)]{Col}
R.~F. Coleman.
\newblock Effective {C}habauty.
\newblock \emph{Duke Math. J.}, 52\penalty0 (3):\penalty0 765--770, 1985.
\newblock ISSN 0012-7094.
\newblock \doi{10.1215/S0012-7094-85-05240-8}.

\bibitem[D\'{e}glise and Nizio{\l}(2018)]{DN18}
F.~D\'{e}glise and W.~Nizio{\l}.
\newblock On {$p$}-adic absolute {H}odge cohomology and syntomic coefficients.
  {I}.
\newblock \emph{Comment. Math. Helv.}, 93\penalty0 (1):\penalty0 71--131, 2018.
\newblock ISSN 0010-2571.
\newblock \doi{10.4171/CMH/430}.

\bibitem[Deligne(1989)]{Deligne89}
P.~Deligne.
\newblock Le groupe fondamental de la droite projective moins trois points.
\newblock In \emph{Galois Groups over {${\bf Q}$} ({B}erkeley, {CA}, 1987)},
  volume~16 of \emph{Math. Sci. Res. Inst. Publ.}, pages 79--297. Springer, New
  York, 1989.
\newblock \doi{10.1007/978-1-4613-9649-9_3}.

\bibitem[Dogra(2020)]{Dogra-unlikely-intersections}
N.~Dogra.
\newblock Unlikely intersections and the {C}habauty--{K}im method over number
  fields.
\newblock \emph{ArXiv e-prints}, 2020.
\newblock arXiv:1903.05032v3 [math.NT].

\bibitem[Faltings(1983)]{Faltings83}
G.~Faltings.
\newblock Endlichkeitss\"{a}tze f\"{u}r abelsche {V}ariet\"{a}ten \"{u}ber
  {Z}ahlk\"{o}rpern.
\newblock \emph{Invent. Math.}, 73\penalty0 (3):\penalty0 349--366, 1983.
\newblock ISSN 0020-9910.
\newblock \doi{10.1007/BF01388432}.

\bibitem[Faltings(1984)]{Faltings84}
G.~Faltings.
\newblock Erratum: ``{F}initeness theorems for abelian varieties over number
  fields''.
\newblock \emph{Invent. Math.}, 75\penalty0 (2):\penalty0 381, 1984.
\newblock ISSN 0020-9910.
\newblock \doi{10.1007/BF01388572}.

\bibitem[Flynn and Wetherell(1999)]{FW99}
E.~V. Flynn and J.~L. Wetherell.
\newblock Finding rational points on bielliptic genus 2 curves.
\newblock \emph{Manuscripta Math.}, 100\penalty0 (4):\penalty0 519--533, 1999.
\newblock ISSN 0025-2611.
\newblock \doi{10.1007/s002290050215}.

\bibitem[Fontaine(1994)]{Fontaine94}
J.-M. Fontaine.
\newblock Le corps des p\'{e}riodes {$p$}-adiques.
\newblock In \emph{P\'{e}riodes $p$-adiques (Bures-sur-Yvette, 1988)}, number
  223 in Ast\'{e}risque, pages 59--111. Soci\'{e}t\'{e} Math\'{e}matique de
  France, Paris, 1994.
\newblock With an appendix by Pierre Colmez.

\bibitem[Grothendieck(1965)]{EGA-IV.2}
A.~Grothendieck.
\newblock \'{E}l\'{e}ments de g\'{e}om\'{e}trie alg\'{e}brique. {IV}. \'{E}tude
  locale des sch\'{e}mas et des morphismes de sch\'{e}mas. {II}.
\newblock \emph{Inst. Hautes \'{E}tudes Sci. Publ. Math.}, \penalty0
  (24):\penalty0 231, 1965.
\newblock ISSN 0073-8301.

\bibitem[Hain and Matsumoto(2003)]{HM03}
R.~Hain and M.~Matsumoto.
\newblock Weighted completion of {G}alois groups and {G}alois actions on the
  fundamental group of {$\mathbb{P}^1-\{0,1,\infty\}$}.
\newblock \emph{Compositio Math.}, 139\penalty0 (2):\penalty0 119--167, 2003.
\newblock ISSN 0010-437X.
\newblock \doi{10.1023/B:COMP.0000005077.42732.93}.

\bibitem[Hast(2021)]{Hast-transcendence}
D.~R. Hast.
\newblock Functional transcendence for the unipotent albanese map.
\newblock \emph{ArXiv e-prints}, 2021.
\newblock arXiv:1911.00587v3 [math.NT]. To appear in \emph{Algebra \& Number
  Theory}.

\bibitem[Jannsen(1989)]{Jannsen89}
U.~Jannsen.
\newblock On the {$l$}-adic cohomology of varieties over number fields and its
  {G}alois cohomology.
\newblock In \emph{Galois groups over {${\bf Q}$} ({B}erkeley, {CA}, 1987)},
  volume~16 of \emph{Math. Sci. Res. Inst. Publ.}, pages 315--360. Springer,
  New York, 1989.
\newblock \doi{10.1007/978-1-4613-9649-9_5}.

\bibitem[Katz and Zureick-Brown(2013)]{KZB13}
E.~Katz and D.~Zureick-Brown.
\newblock The {C}habauty--{C}oleman bound at a prime of bad reduction and
  {C}lifford bounds for geometric rank functions.
\newblock \emph{Compos. Math.}, 149\penalty0 (11):\penalty0 1818--1838, 2013.
\newblock ISSN 0010-437X.
\newblock \doi{10.1112/S0010437X13007410}.

\bibitem[Katz et~al.(2016)Katz, Rabinoff, and Zureick-Brown]{KRZB}
E.~Katz, J.~Rabinoff, and D.~Zureick-Brown.
\newblock Uniform bounds for the number of rational points on curves of small
  {M}ordell-{W}eil rank.
\newblock \emph{Duke Math. J.}, 165\penalty0 (16):\penalty0 3189--3240, 2016.
\newblock ISSN 0012-7094.
\newblock \doi{10.1215/00127094-3673558}.

\bibitem[Kim(2005)]{Kim-siegel}
M.~Kim.
\newblock The motivic fundamental group of the projective line minus three
  points and the theorem of {S}iegel.
\newblock \emph{Invent. Math.}, 161\penalty0 (3):\penalty0 629--656, 2005.
\newblock ISSN 0020-9910.
\newblock \doi{10.1007/s00222-004-0433-9}.

\bibitem[Kim(2009)]{Kim-alb}
M.~Kim.
\newblock The unipotent {A}lbanese map and {S}elmer varieties for curves.
\newblock \emph{Publ. Res. Inst. Math. Sci.}, 45\penalty0 (1):\penalty0
  89--133, 2009.
\newblock ISSN 0034-5318.
\newblock \doi{10.2977/prims/1234361156}.

\bibitem[Kim(2012)]{Kim-galois}
M.~Kim.
\newblock Galois theory and {D}iophantine geometry.
\newblock In \emph{Non-Abelian Fundamental Groups and {I}wasawa Theory}, volume
  393 of \emph{London Math. Soc. Lecture Note Ser.}, pages 162--187. Cambridge
  Univ. Press, Cambridge, 2012.

\bibitem[McCallum and Poonen(2012)]{MP12}
W.~McCallum and B.~Poonen.
\newblock The method of {C}habauty and {C}oleman.
\newblock In \emph{Explicit Methods in Number Theory}, volume~36 of
  \emph{Panor. Synth\`eses}, pages 99--117. Soc. Math. France, Paris, 2012.

\bibitem[Milne(2006)]{Milne-ADT}
J.~S. Milne.
\newblock \emph{Arithmetic Duality Theorems}.
\newblock BookSurge, LLC, Charleston, SC, second edition, 2006.
\newblock ISBN 1-4196-4274-X.

\bibitem[Milne(2017)]{Milne-AG}
J.~S. Milne.
\newblock \emph{Algebraic Groups}, volume 170 of \emph{Cambridge Studies in
  Advanced Mathematics}.
\newblock Cambridge University Press, Cambridge, 2017.
\newblock ISBN 978-1-107-16748-3.
\newblock \doi{10.1017/9781316711736}.

\bibitem[Neukirch et~al.(2008)Neukirch, Schmidt, and Wingberg]{NSW08}
J.~Neukirch, A.~Schmidt, and K.~Wingberg.
\newblock \emph{Cohomology of number fields}, volume 323 of \emph{Grundlehren
  der Mathematischen Wissenschaften [Fundamental Principles of Mathematical
  Sciences]}.
\newblock Springer-Verlag, Berlin, second edition, 2008.
\newblock ISBN 978-3-540-37888-4.
\newblock \doi{10.1007/978-3-540-37889-1}.

\bibitem[Poonen(2005)]{P}
B.~Poonen.
\newblock Unramified covers of {G}alois covers of low genus curves.
\newblock \emph{Math. Res. Lett.}, 12\penalty0 (4):\penalty0 475--481, 2005.
\newblock ISSN 1073-2780.
\newblock \doi{10.4310/MRL.2005.v12.n4.a3}.

\bibitem[Serre(1989)]{Serre-MW}
J.-P. Serre.
\newblock \emph{Lectures on the {M}ordell-{W}eil theorem}, volume~15 of
  \emph{Aspects of Mathematics}.
\newblock Friedr. Vieweg \& Sohn, Braunschweig, 1989.
\newblock ISBN 3-528-08968-7.
\newblock \doi{10.1007/978-3-663-14060-3}.
\newblock Translated from the French and edited by Martin Brown from notes by
  Michel Waldschmidt.

\bibitem[Stoll(2006)]{Stoll06}
M.~Stoll.
\newblock Independence of rational points on twists of a given curve.
\newblock \emph{Compos. Math.}, 142\penalty0 (5):\penalty0 1201--1214, 2006.
\newblock ISSN 0010-437X.
\newblock \doi{10.1112/S0010437X06002168}.

\bibitem[Stoll(2019)]{Stoll}
M.~Stoll.
\newblock Uniform bounds for the number of rational points on hyperelliptic
  curves of small {M}ordell-{W}eil rank.
\newblock \emph{J. Eur. Math. Soc. (JEMS)}, 21\penalty0 (3):\penalty0 923--956,
  2019.
\newblock ISSN 1435-9855.
\newblock \doi{10.4171/JEMS/857}.

\bibitem[Voloch(2012)]{Voloch12}
J.~F. Voloch.
\newblock Finite descent obstruction for curves over function fields.
\newblock \emph{Bull. Braz. Math. Soc. (N.S.)}, 43\penalty0 (1):\penalty0 1--6,
  2012.
\newblock ISSN 1678-7544.
\newblock \doi{10.1007/s00574-012-0001-7}.

\end{thebibliography}

\end{document}